\documentclass[english]{smfart}
\usepackage{smfthm}
\usepackage{amsmath}            
\usepackage{amsbsy}      
\usepackage{amscd}      
\usepackage[dvips]{graphicx}
\usepackage[all]{xy}
\usepackage{amssymb}      

\theoremstyle{plain}      

\newtheorem{theorem}{Theorem}[section]      
\newtheorem{lemma}{Lemma}[section]

\newtheorem{proposition}{Proposition}[section]

\theoremstyle{remark}      
\newtheorem{remark}{Remark}[section]

\newcommand{\Z}{{\mathbb{Z}}}   
  
\newcommand{\C}{{\mathbb{C}}}

\renewcommand{\O}{{\mathcal{O}}}          
\newcommand{\q}{{\mathfrak{q}}}    
\newcommand{\A}{{\mathfrak{A}}}

\author{Louis Funar}
\address{Institut Fourier BP 74, UMR 5582, Laboratoire de Math\'ematiques, 
Universit\'e Grenoble Alpes, CS 40700, 38058 Grenoble cedex 9, France}
\email{louis.funar@univ-grenoble-alpes.fr}

\author{Wolfgang Pitsch}
\address{Departament de Matem\`atiques \\  Universitat Aut\`onoma de Barcelona \\ 08193 Bellaterra (Cerdanyola del Vall\`es), Espa\~na}
\email{pitsch@mat.uab.es}

\title[Symplectic Schur multiplier]{The Schur multiplier of finite symplectic groups}
\alttitle{Multiplicateur de Schur des groupes symplectiques finis}

\begin{document}

\frontmatter

\begin{abstract}
We show that the Schur multiplier of 
$Sp(2g,\Z/D\Z)$ is $\Z/2\Z$, when $D$ is divisible by 4. 
\end{abstract}

\begin{altabstract}
Nous montrons que le multiplicateur de Schur de $Sp(2g,\Z/D\Z)$ est $\Z/2\Z$ quand $D$ est divisible par 4. 

\end{altabstract}

\subjclass{57 M 50, 55 N 25, 19 C 09, 20 F 38}

\thanks{The first author was supported by the ANR 2011 BS 01 020 01 ModGroup and the second author by the FEDER/MEC grant MTM2010-20692.}

\maketitle
\mainmatter

\section{Introduction and statements}

Let $g \geq 1$ be an integer and denote  by $Sp(2g,\Z)$ the symplectic group of $2g \times 2g$ 
matrices with integer coefficients.  Deligne's non-residual finiteness theorem from \cite{De} states that the {\em universal central extension} $\widetilde{Sp(2g,\Z)}$ is 
not residually finite  since the image of its center  under any homomorphism into 
a finite group has order at most two when $g\geq 3$. Our first motivation was to understand this result and give a sharp statement, namely to decide whether the image of the central $\Z$  
might be of order two. 
Since these symplectic groups have the  congruence subgroup property, this boils down to understanding the second homology of symplectic groups with coefficients in finite cyclic groups. In the sequel, for simplicity and unless otherwise explicitly stated, all (co)homology groups will be understood to be with trivial integer coefficients. An old theorem of Stein (see \cite{Stein}, Thm. 2.13 and Prop. 3.3.a) is that 
 $H_2(Sp(2g,\Z/D\Z))=0$, when $D$ is not divisible by 4. 
The case $D\equiv  0 \; ({\rm mod} \; 4)$ remained open since then; this is explicitly mentioned for instance in (\cite{Pu1}, Remarks after Thm. 3.8).  Our main result settles this case: 

\begin{theorem}\label{tors-sympl0}
The second homology group of finite principal congruence quotients 
of $Sp(2g,\Z)$, $g\geq 3$ is 
\[H_2(Sp(2g,\Z/D\Z))=
\Z/2\Z, \;  {\rm if }\; D\equiv 0 \; ({\rm mod} \; 4). 
\]  
\end{theorem}

In comparison, recall that Beyl (see \cite{Beyl}) has showed that  $H_2(SL(2,\Z/D\Z))=\Z/2\Z$, for 
$D\equiv 0\,({\rm mod }\; 4)$ and 
Dennis and Stein proved using K-theoretic methods that for $n\geq 3$ we have 
$H_2(SL(n,\Z/D\Z))=\Z/2\Z$, for 
$D\equiv 0\,({\rm mod }\; 4)$, while $H_2(SL(n,\Z/D\Z))=0$,  for 
$D\not\equiv 0\,({\rm mod }\; 4)$ (see \cite{DS1}, Cor. 10.2 and \cite{Milnor}, section 12).

Our proof also relies on Deligne's non-residual finiteness theorem from \cite{De} and 
deep results of Putman in \cite{Pu1}, and shows that we can 
detect this $\Z/2\Z$ factor on $H_2(Sp(2g, \Z/32\Z))$ for $g \geq 4$, 
providing an explicit extension that detects this homology class.

{\bf Acknowledgements.} We are thankful to  Jean Barge, Nicolas Bergeron, Dave Benson, 
Will Cavendish, Florian Deloup, Philippe Elbaz-Vincent, Richard Hain, Greg McShane, Ivan Marin, 
Gregor Masbaum, Alexander Rahm and Alan Reid for helpful 
discussions and suggestions.  We are grateful to Pierre Lochak 
and Andy Putman for their help in clarifying a number of 
technical points and  improving the presentation and the referee for cleaning and symplifying our proof. 

\section{Preliminaries}
\subsection{Residual finiteness of universal central extensions}

In this section we collect results about universal central extensions of perfect groups,  for the sake of completeness of our arguments.  Every perfect group $\Gamma$  has a universal central extension $\widetilde{\Gamma}$; the kernel of the canonical projection map  $\widetilde{\Gamma} \to \Gamma$ contains the center $Z(\widetilde{\Gamma})$ of 
$\widetilde{\Gamma}$ and is canonically isomorphic to the second integral homology group $H_2(\Gamma)$.  
We  will recall now how  the residual finiteness problem for the universal central $\widetilde{\Gamma}$ of a perfect and residually finite group $\Gamma$  translates into an homological 
problem about $H_2(\Gamma)$. 
We start with a classical result for maps between universal  central extensions of perfect groups. 
\begin{lemma}\label{lift}
Let $\Gamma$ and $F$ be perfect groups, $\widetilde{\Gamma}$ and 
$\widetilde{F}$ be their universal central extensions and 
$p: \Gamma \rightarrow F$ be a group homomorphism. Then there 
exists a unique homomorphism 
$\widetilde{p}:\widetilde{\Gamma}\to \widetilde{F}$ lifting $p$ such 
that the following diagram is commutative: 
\[ \begin{array}{clcrclc}
1 \to& H_2(\Gamma)&\to& \widetilde{\Gamma}&\to& \Gamma& \to 1 \\
     & p_*\downarrow & & \widetilde{p} \downarrow & & \downarrow p & \\
1 \to& H_2(F)&\to& \widetilde{F}&\to& F & \to 1  \\
\end{array}
\]
\end{lemma}

For a proof we refer the interested reader to (\cite{Lang}, chap VIII) or (\cite{Brown}, chap IV, Ex. 1, 7). If $\Gamma$ is a perfect  residually finite group, to prove  that its universal central extension 
$\widetilde{\Gamma}$ is also residually finite we only have to find enough finite quotients of $\widetilde{\Gamma}$ to detect the elements in its center $H_2(\Gamma)$. The following lemma analyses the situation.
  
\begin{lemma}\label{kernel}
Let $\Gamma$ be a  perfect group and denote by
$\widetilde{\Gamma}$ its universal central extension. 
 
\begin{enumerate}
\item Let $H$ be a  finite index normal subgroup 
$H\subset \Gamma$ such that 
the image of $H_2(H)$ into $H_2(\Gamma)$ contains the subgroup $dH_2(\Gamma)$, for some $d\in\Z$. 
Let $F=\Gamma/H$ be the corresponding finite quotient 
of $\Gamma$ and $p:\Gamma\to F$ the quotient map. 
Then $d\cdot p_*(H_2(\Gamma))=0$, where $p_*:H_2(\Gamma)\to H_2(F)$ 
is the homomorphism induced by $p$. 
In particular, if $p_*:H_2(\Gamma)\to H_2(F)$ is surjective, then 
$d\cdot H_2(F)=0$. 
\item Assume that $F$ is a finite quotient of $\Gamma$ satisfying  
$d\cdot p_*(H_2(\Gamma))=0$. Let $\widetilde{F}$ denote the universal central extension of $F$. 
Then the homomorphism $p:\Gamma\to F$ has a unique lift  
$\widetilde{p}:\widetilde{\Gamma}\to \widetilde{F}$ and the kernel of 
$\widetilde{p}$  contains $d\cdot H_2(\Gamma) $.
\end{enumerate} 
\end{lemma}
Observe that in point $2.$ of Lemma \ref{kernel} the group $F$ being finite, $H_2(F)$ is also finite, hence 
one can take $d= |H_2(F)|$.

\begin{proof} 
The image of $H$ into $F$ is trivial and thus 
the image of $H_2(H)$ into $H_2(F)$ is trivial. 
This implies that $p_*(d\cdot H_2(\Gamma))=0$,
which proves the first part of the lemma.

Further, by Lemma \ref{lift} there exists an unique lift 
$\widetilde{p}:\widetilde{\Gamma}\to \widetilde{F}$. 
If $d\cdot p_*(H_2(\Gamma))=0$ then Lemma \ref{lift} yields  
$d\cdot \widetilde{p}(c)=d \cdot p_*(c)=0$, 
for any $c\in H_2(\Gamma)$. This settles the second part of the lemma.  
\end{proof}

\begin{remark}\label{divisors}
It might be possible that we have  $d'\cdot p_*(H_2(\Gamma))=0$ for some 
proper divisor $d'$ of $d$, so the first part of Lemma \ref{kernel} 
can only give an upper bound of the orders of the image 
of the second cohomology. In order to find lower bounds we need additional 
information concerning the finite quotients $F$. 
\end{remark}

\begin{lemma}\label{key lemma}
Let $\Gamma$ be a perfect group, 
$\widetilde{\Gamma}$  its universal central extension, 
$p:\Gamma\to F$ be a surjective homomorphism onto a finite group $F$ and 
$\widehat{p}:\widetilde{\Gamma}\to G$ be some lift of $p$ to 
a central extension $G$ of $F$ by some finite abelian group $C$. 
Assume that the image 
of $H_2(\Gamma)\subset \widetilde{\Gamma}$ in $G$ by $\widehat{p}$ contains an element of order $q$. Then there exists an element 
of $p_*(H_2(\Gamma))\subset H_2(F)$ of order $q$. 
\end{lemma}
\begin{proof}
 By Lemma \ref{lift} there exists a lift 
$\widetilde{p}: \widetilde{\Gamma}\to \widetilde{F}$ of $p$ into 
the universal central extension $\widetilde{F}$ of $F$. Then, by universality 
there exists a unique homomorphism $s:\widetilde{F}\to G$
of central extensions of $F$  lifting the identity map of $F$.  
The homomorphisms $\widehat{p}$ and 
$s\circ \widetilde{p}:\widetilde{\Gamma}\to G$ are then both lifts of $p$.  Using the centrality of $C$ in $G$  
it follows that the map $\widetilde{\Gamma} \to C$ given by $x \mapsto \widehat{p}(x)^{-1} \cdot (s \circ \widetilde{p} (x))$ is a group homomorphism, and hence is trivial since $\widetilde{\Gamma}$ is perfect and $C$ abelian. We conclude that $\widehat{p}= s\circ \widetilde{p}$. 

Recall that the  restriction of $\widetilde{p}$ 
to $H_2(\Gamma)$ coincides with the homomorphism 
$p_*:H_2(\Gamma)\to H_2(F)$ and that $H_2(F)$ is finite since $F$ is so.  Then, if $z\in H_2(\Gamma)$ is such that $\widehat{p}(z)$ has order $q$ in $C$, 
the element $p_*(z)\in p_*(H_2(\Gamma))\subset \widetilde{F}$ is sent by $s$ 
onto an element of order $q$ and therefore $p_*(z)$ has order a multiple 
of $q$, say $aq$. Then  $(p_*(z))^a = p_\ast(z^a) \in p_*(H_2(\Gamma))\subset \widetilde{F}$ has  order $q$. 
\end{proof}


\section{Proof of Theorem \ref{tors-sympl0}}\label{compute}

\subsection{Reducing the proof to $D=2^k$} 
Let $D=p_1^{n_1}p_2^{n_2}\cdots p_s^{n_s}$ be the prime decomposition of 
an integer $D$.  Then, according to  (\cite[Thm. 5]{NS}) we have 
$Sp(2g,\Z/D\Z)=\oplus_{i=1}^s Sp(2g,\Z/p_i^{n_i}\Z)$.
Since symplectic groups are perfect for $g\geq3$ (see e.g. 
\cite{Pu1}, Thm. 5.1),  from the K\"unneth formula, we  derive: 
\[
H_2(Sp(2g,\Z/D\Z))=\oplus_{i=1}^s H_2(Sp(2g,\Z/p_i^{n_i}\Z)).
\] 
Stein (see \cite{Stein,Stein2}) proved that for any odd prime $p$ and $n\geq 2$ the Schur multipliers vanish: 
\[
H_2(Sp(2g,\Z/p^{n}\Z))=0
\] 
while Steinberg has showed that 
\[
H_2(Sp(2g,\Z/2\Z))=0. 
\] 
Then, Theorem \ref{tors-sympl0} is 
equivalent to the  statement: 
\[ H_2(Sp(2g,\Z/2^k\Z)) = \Z/2\Z, \; {\rm for \; all \; } g \geq 3, k \geq 2.  
\]

We will freely use in the sequel two classical results due to Stein.   
{\em Stein's isomorphism theorem} (see \cite{Stein}, Thm. 2.13 and Prop. 3.3.(a))
states that there is an isomorphism:  
\[
H_2(Sp(2g,\Z/2^k\Z))\simeq H_2(Sp(2g,\Z/2^{k+1}\Z)), \; 
{\rm for \; all \; } g \geq 3, k \geq 2. 
\]

Further, {\em Stein's stability theorem} (see \cite{Stein}) states that 
the stabilization homomorphism $Sp(2g,\Z/2^k\Z)\hookrightarrow 
Sp(2g+2,\Z/2^k\Z)$ induces an isomorphism: 
\[
H_2(Sp(2g,\Z/2^k\Z))\simeq H_2(Sp(2g+2,\Z/2^{k}\Z)), \; 
{\rm for \; all \; } g \geq 3, k \geq 1. 
\]
Therefore, to prove Theorem \ref{tors-sympl0}  it suffices to show that: 
\[ H_2(Sp(2g,\Z/2^k\Z))=\Z/2\Z, \; {\rm for \; some \; } g \geq 3 \; {\rm and \; some \; } k \geq 2. 
\]

\subsection{An alternative for the order of $H_2(Sp(2g,\Z/2^k\Z))$}

As our first step we prove, as a consequence of Deligne's theorem: 

\begin{proposition}\label{alternative}
We have $H_2(Sp(2g,\Z/2^k\Z))\in \{0, \Z/2\Z\}$, when $g\geq 4$.
\end{proposition}

 \begin{proof}[Proof of Proposition \ref{alternative}]
Let $p:Sp(2g,\Z)\to Sp(2g,\Z/2^k\Z)$ be the reduction mod $2^k$
and $p_*:H_2(Sp(2g,\Z))\to H_2(Sp(2g,\Z/2^k\Z))$ the induced homomorphism.   
The first ingredient in the proof is the following result which seems well-known, and that we isolate for later reference: 

\begin{lemma}\label{surj}
The homomorphism $p_*:H_2(Sp(2g,\Z))\to H_2(Sp(2g,\Z/2^k\Z))$ is surjective, if $g\geq 4$. 
\end{lemma}
\noindent The rather technical proof of Lemma \ref{surj} is postponed to section \ref{surjective}.

Now, it is a classical result that $H_1(Sp(2g,\Z))=0$, for 
$g\geq 3$ and $H_2(Sp(2g,\Z))=\Z$, for $g\geq 4$ (see e.g. 
\cite{Pu1}, Thm. 5.1). Note however that for $g=3$, 
we have  that $H_2(Sp(6,\Z))=\Z\oplus \Z/2\Z$ according to \cite{Stein3}. 
This implies that 
$H_2(Sp(2g,\Z/2^k\Z))$ is cyclic when $g\geq 4$ (this was also shown by Stein in \cite{Stein}) and we only have to bound the order of this cohomology group.

Lemma \ref{lift} provides a lift between the universal central extensions $\widetilde{p}:\widetilde{Sp(2g,\Z)}\to 
\widetilde{Sp(2g,\Z/2^k\Z)}$ of the mod $2^k$ reduction map, such that the restriction of $\widetilde{p}$ 
to the central subgroup $H_2(Sp(2g,\Z))$ of $\widetilde{Sp(2g,\Z)}$ is 
the homomorphism $p_*:H_2(Sp(2g,\Z))\to H_2(Sp(2g,\Z/2^k\Z))$.  From Deligne's theorem \cite{De}  every 
finite index subgroup of the universal central extension  
$\widetilde{Sp(2g,\Z)}$, for $g\geq 4$, contains $2\Z$, where 
$\Z$ is the central kernel  $\ker(\widetilde{Sp(2g,\Z)}\to Sp(2g,\Z))$. If $c$ is a generator of 
the central  $\Z=H_2(Sp(2g,\Z))$ we have  $2p_*(c)=\widetilde{p}(2c)=0$. 
According to Lemma \ref{surj} $p_*$ is surjective and thus  
$H_2(Sp(2g,\Z/2^k\Z))$ is  a quotient of $\mathbb{Z}/2\mathbb{Z}$, as claimed.  
\end{proof}

\subsection{Divisibility of the universal symplectic central extension when restricted to level
subgroups}
The Siegel upper half plane $\mathcal H_g$ is the space of $g\times g$ symmetric complex matrices 
with positive definite imaginary parts.
Let $Sp(2g,L)$ denote the {\em level $L$ congruence subgroup} of $Sp(2g,\Z)$, namely the 
kernel of the mod $L$ reduction map $Sp(2g,\Z)\rightarrow Sp(2g,\Z/L\Z)$. 
The moduli space $\mathcal A_g(L)$ of principally polarized 
abelian varieties with level $L$ structures (over $\C$ of dimension $g$) is defined as the quotient 
$\mathcal H_g/Sp(2g,L)$. This is a quasiprojective orbifold. 
As $Sp(2g,\Z)$ has Kazhdan's property $T$ for $g\geq 2$ we have  
$H^1(Sp(2g,L);\Z)=0$ and there exists an injection of the Picard group 
$Pic(\mathcal A_g)\to H^2(Sp(2g,L);\Z)$; Borel (\cite{Bo}) proved 
that $H^2(Sp(2g,L);\Z)$ has rank $1$, for $g\geq 3$ and Putman (\cite{Pu1}, Thm. D) showed that 
for $g\geq 4$ and $4\nmid L$ this injection is an isomorphism. 
There is a line bundle $\lambda_g\in Pic(\mathcal A_g)$ whose holomorphic sections are 
Siegel modular forms of weight $1$ and level $1$, also called the Hodge line bundle.  
Then $\lambda_g$ generates $Pic(\mathcal A_g)$ (see \cite{Frei}) and 
its first Chern class $c_1(\lambda_g)$ generates $H^2(Sp(2g,\Z);\Z)=\Z$, for $g\geq 3$. 
Denote by $\lambda_g(L)\in Pic(\mathcal A_g(L))$ the pullback of $\lambda_g$ to the orbifold covering $\mathcal A_g(L)$.

Recall the (slightly corrected) version of Putman's lemma (\cite{Pu1}, Lemma 5.5):
\begin{lemma}
Let $L \geq 2$ be an even number. The pull-back $\lambda_g(L)$ of the Hodge bundle $\lambda_g \in \mathcal{A}_g$ to $\mathcal{A}_g(L)$ is divisible by $2$ if  $4 \mid L$ and 
is divisible by $2$ modulo torsion but not divisible by $2$ if $4\nmid L$.   
\end{lemma}

In view of the canonical injection $Pic(\mathcal{A}_g(L)) \hookrightarrow H^2(Sp(2g,L);\mathbb{Z})$, this result cohomologically translates as:

\begin{lemma}\label{putmanlemma}
Let $L\geq 2$ be an even number. The pullback of the class of the universal central extension $c \in H^2(Sp(2g,\Z);\Z)$ to $Sp(2g,L)$ is
divisible by $2$ if $4 \mid L$ and 
 divisible by $2$ modulo torsion if $4 \nmid L$.  
\end{lemma}
\begin{proof}
The transformation formulas for the theta nulls (see e.g. \cite{Frei}) provide a square root for the pull-back, nevertheless, these computations do not see the torsion (i.e. flat) bundles. So they provide  for $L=2$ and hence for every even $L$  a square root for $\lambda_g$ modulo torsion.

The universal coefficients exact sequence reads:  
\[
0 \to  {\rm Hom}(H_1(Sp(2g,2);\mathbb{Z}),\mathbb{C}^\ast) \to  H^2(Sp(2g,2);\mathbb{Z})
 \to {\rm Hom}(H_2(Sp(2g,2);\mathbb{Z}),\mathbb{Z}) \to 0\]

As $Sp(2g,\mathbb{Z})$ is perfect, the  divisibility by 2 mod torsion of  $\lambda_g(2)$ shows that the image of $H_2(Sp(2g,2);\mathbb{Z}) \to H_2(Sp(2g,\mathbb{Z});\mathbb{Z}) =\mathbb{Z}$ is contained in $2\mathbb{Z}$. By Deligne's theorem,  the image must in fact be $2\mathbb{Z}$.

Further, torsion elements in $H^2(Sp(2g,2);\Z)$ all come from the abelianization of $Sp(2g,2)$.  It is proved in \cite{Igusa} that the commutator $[Sp(2g,2),Sp(2g,2)]$ coincides 
with the so-called Igusa subgroup $Sp(2g,4,8)$ of $Sp(2g,4)$ consisting of those symplectic matrices 
$\left(\begin{array}{cc}
A & B \\ 
C & D
\end{array}
\right)$ with the property that the diagonal entries of 
$A B^{\top}$ and  $C D^{\top}$  are multiples of $8$.
Since $Sp(2g,4,8)\supset Sp(2g,8)$ 
the pull-back of the universal central extension  on $Sp(2g,\mathbb{Z})$ to $Sp(2g,8)$ is genuinely divisible by $2$.

We are only left with the proof that the pullback of the class $c$ on $Sp(2g,4)$ is divisible by $2$. This is a consequence of Stein's stability theorem. Indeed, the divisibility by $2$ on $Sp(2g,8)$ implies that the mod $2$ reduction of the universal central extension $c$ is the pullback of the universal central extension of $Sp(2g,\Z/32\Z)$, so by stability  it is also the pullback of the universal central extension of $Sp(2g,\Z/4\Z)$, and we have a commutative diagram:

\begin{equation}
\begin{array}{clccccc}
1\to &\Z     & \to& \widetilde{Sp(2g,\Z)} & \to &  Sp(2g,\Z)& \to 1\\
     & \downarrow {\rm mod} \, 2 & & \downarrow F & & \downarrow  & \\
1\to &\Z/2\Z& \to &\widetilde{Sp(2g,\Z/4\Z)} &\to& Sp(2g,\Z/4\Z)       & \to 1.
\end{array}
\end{equation}
Let $H= \ker F$. Then $H$ is the preimage of $Sp(2g,4)\subset Sp(2g,\Z)$ and $H \cap \Z=2\Z$. 
By construction twice the class of the central extension 
\[ 1 \to \Z \to H \to Sp(2g,4)\to 1\]
is $c$ restricted to $Sp(2g,4)$.
\end{proof}

\subsection{Nonsplit central extensions}
In order to show that $H_2(Sp(2g,\Z/2^k\Z))=\Z/2\Z$, when $k\geq 2$ and $g\geq 4$ it is enough to 
show that $H_2(Sp(2g,\Z/m\Z))\neq 0$, for some value of $m$. 
This will follow by exhibiting a nonsplit central extension of $Sp(2g,\Z/m\Z)$. 
Let 
\[ 1\to \Z \to \widetilde{Sp(2g,\Z)} \to Sp(2g,\Z)\to 1\]
be the universal central extension of $Sp(2g,\Z)$ and let 
$c\in H^2(Sp(2g,\Z))=\Z$ be the generator. 
Let $Sp(2g,2)$ be the level 2 subgroup of $Sp(2g,\Z)$, namely the 
kernel of the mod $2$ reduction map $Sp(2g,\Z)\rightarrow Sp(2g,\Z/2\Z)$ and 
$\iota: Sp(2g,2) \to Sp(2g,\Z)$ the inclusion. 

 From  Lemma \ref{putmanlemma} there exists some $d\in H^2(Sp(2g,4))$
 such that $2d=\iota^*(c)$. 
 Let 
 \[ 1\to \Z \to \widetilde{G} \to Sp(2g,4)\to 1\]
be the central extension corresponding to the class $d$, so that we have a commutative diagram: 
\begin{equation}
\begin{array}{clccccc}
1\to &\Z& \to &\widetilde{G}&\to& Sp(2g,4)& \to 1\\
 & \downarrow \times 2 & & \downarrow & & \downarrow \iota & \\
1\to &\Z& \to &\widetilde{Sp(2g,\Z)} &\to& Sp(2g,\Z)& \to 1.
\end{array}
\end{equation}
The group $\widetilde{G}$ is thus a finite-index subgroup of $\widetilde{Sp(2g,\Z)}$. 
Let $\widetilde{K}\subset \widetilde{G}$ be a finite-index normal subgroup of 
$\widetilde{Sp(2g,\Z)}$ and $K$ be its image in $Sp(2g,4)$. 

Using the congruence subgroup property, we can find some $m$ such that 
$Sp(2g,m)\subset K$. 
Let $\widetilde{Sp(2g,m)}$ be the preimage of $Sp(2g,m)$ under the map 
$\widetilde{Sp(2g,\Z)} \to Sp(2g,\Z)$ and 
$\widetilde{H}=\widetilde{K}\cap \widetilde{Sp(2g,m)}$. 
Since it is the intersection of two finite-index normal subgroups, the group $\widetilde{H}$ 
is a finite-index normal subgroup of $\widetilde{Sp(2g,\Z)}$ that is contained in 
$\widetilde{K}$. 
Its intersection with the central $\Z$ in $\widetilde{Sp(2g,\Z)}$ is thus of the form 
$2n\Z$, for some $n\geq 1$. By Deligne's theorem we derive that $n=1$. 
We obtain the commutative diagram: 
\begin{equation}
\begin{array}{clccccc}
1\to &\Z& \to &\widetilde{H}&\to& Sp(2g,m)& \to 1\\
 & \downarrow \times 2 & & \downarrow & & \downarrow \iota & \\
1\to &\Z& \to &\widetilde{Sp(2g,\Z)} &\to& Sp(2g,\Z)& \to 1.
\end{array}
\end{equation}
We have $Sp(2g,\Z/m\Z)=Sp(2g,\Z)/Sp(2g,m)$. Define $\Gamma=\widetilde{Sp(2g,\Z)}/\tilde{H}$. 
We thus have a central extension
\[ 1\to \Z/2\Z \to \Gamma \to Sp(2g,\Z/m\Z)\to 1\]
Since $\widetilde{Sp(2g,\Z)}$ is the universal central extension of the perfect group 
$Sp(2g,\Z)$ the group $\widetilde{Sp(2g,\Z)}$  is perfect. This implies that 
$\Gamma$ is also perfect, and in particular it has no nontrivial homomorphism to $\Z/2$. 
We conclude that this central extension of $Sp(2g,\Z/m\Z)$ does not split, as desired. 

\begin{remark}
We can take for $\widetilde{K}=[\widetilde{Sp(2g,4)},\widetilde{Sp(2g,4)}]$, 
which is a finite index normal subgroup of $\widetilde{Sp(2g,\Z)}$. 
In this case $K$ is the Igusa  
subgroup $Sp(2g, 16, 32)$  consisting of those symplectic matrices 
$\left(\begin{array}{cc}
A & B \\ 
C & D
\end{array}
\right)$ from $Sp(2g,16)$ with the property that the diagonal entries of 
$A B^{\top}$ and  $C D^{\top}$  are multiples of $32$.
Therefore we can take $m=32$ in the argument above.   
\end{remark}

\begin{remark}
As a consequence Deligne's non-residual finiteness 
theorem is sharp. Putman in (\cite{Pu1}, Thm.F) has previously obtained the existence
of finite index subgroups of $\widetilde{Sp(2g,\Z)}$ which contain
$2\Z$ but not $\Z$. We make explicit his construction as the image of the center of the 
{universal central extension} $\widetilde{Sp(2g,\Z)}$ 
into the universal central extension of $Sp(2g,\Z/D\Z)$ is of order two, when $D\equiv 0 \; ({\rm mod} \; 4)$ and $g\geq 3$. 
\end{remark}

\subsection{Proof of Lemma \ref{surj}}\label{surjective}
Let $\mathbb K$ be a {\em number field}, ${\mathcal R}$ be the set of 
inequivalent valuations of $\mathbb K$ and 
$S\subset \mathcal R$ be a finite set of valuations of $\mathbb K$ 
including all the Archimedean (infinite) ones. 
Let 
\[ O(S)=\{x\in \mathbb K\colon\ v(x)\leq 1, \:\:  
{\rm for \:\: all } \: v\in \mathcal R\setminus S\}\]  
be the ring of $S$-integers in $\mathbb K$ and 
$\q\subset O(S)$ be a nonzero ideal. 
By $\mathbb K_v$ we denote the completion of $\mathbb K$ 
with respect to $v\in \mathcal R$. Following \cite{BMS},
a domain $\A$ which arises as $O(S)$ above will be called a 
{\em Dedekind domain of arithmetic type}. 

\vspace{0.2cm}
\noindent
Let $\A=O_S$ be a Dedekind domain of arithmetic type and 
$\q$ be an ideal of $\A$.  Denote by $Sp(2g,\A,\q)$ the kernel of the 
surjective homomorphism $p:Sp(2g,\A) \to Sp(2g, \A/\q)$. 
 The surjectivity is not a purely formal fact and  follows from 
the fact that in these cases the symplectic group coincides with 
the so-called "elementary symplectic group", and that it is trivial 
to  lift elementary generators of $Sp(2g, \A/\q)$ to $Sp(2g, \A)$, 
for a proof of this fact when $\A=\Z$ see (\cite{HO}, Thm. 9.2.5).

The goal of this section is to give a self-contained proof of the following result, which 
contains Lemma \ref{surj} as a particular case:

\begin{proposition}\label{lem H2epi}
Given an ideal $\q \in \A$, for any $g \geq 3$, the homomorphism  $p_*:H_2(Sp(2g,\A)) \to H_2(Sp(2g,\A/\q))$ is 
surjective. 
\end{proposition}

Normal generators of the group $Sp(2g,\A,\q)$ can be found in 
(\cite{BMS}, III.12), as follows.  
Fix a symplectic basis $\{a_i,b_i\}_{1 \leq i \leq g}$ and write 
the matrix by blocks according to the associated decomposition 
of $\A^{2g}$ into maximal isotropic subspaces.

\vspace{0.2cm}
\noindent
For each pair of distinct integers $i,j\in\{1,\dots,g\}$ denote by 
$e_{ij}\in \mathfrak M_{g}(\Z)$ the matrix whose only non-zero entry 
is a $1$ at the place $ij$. Set also $\mathbf 1_{k}$ for 
the $k$-by-$k$ identity matrix.  

\vspace{0.2cm}
\noindent
Then following (\cite{BMS}, Lemma 13.1) $Sp(2g,\A,\q)$ is the {\em normal subgroup} 
of $Sp(2g,\A)$ generated by the matrices: 
\begin{equation}
 U_{ij}(q) = \left( \begin{matrix}
                  \mathbf 1_{g} & qe_{ij}+qe_{ji} \\
                  0 & \mathbf 1_{g}
                 \end{matrix}
 \right), \quad U_{ii}(q)  = \left( \begin{matrix}
                  \mathbf 1_{g} & qe_{ii} \\
                  0 & \mathbf 1_{g}
                 \end{matrix}
 \right),
\end{equation} 
and
\begin{equation}
L_{ij}(q) = \left( \begin{matrix}
                 \mathbf 1_{g} & 0 \\
                  qe_{ij}+qe_{ji}  & \mathbf 1_{g}
                 \end{matrix}
 \right), \quad L_{ii}(q) 
= \left( \begin{matrix}
                  \mathbf 1_{g} & 0 \\
                  qe_{ii}  & \mathbf 1_{g}
                 \end{matrix}
 \right),
\end{equation}
where $q\in \q$. 

\vspace{0.2cm}
\noindent
Denote by $E(2g, \A,\q)$ the subgroup of $Sp(2g,\A,\q)$ generated by the 
matrices $U_{ij}(q)$ and $L_{ij}(q)$. 
\begin{lemma}\label{thetasubgroup}
The group $E(2g,\A,\q)$ is the subgroup $Sp(2g, \A, \q|\q^2)$ of $Sp(2g,\A)$  
of those symplectic matrices $\left( \begin{matrix}
                  A & B \\
                  C  & D
                 \end{matrix}
 \right)$ whose entries satisfy the conditions: 
\begin{equation}
B\equiv C\equiv 0\: ({\rm mod}\: \q), \: {\rm and} \:
A\equiv D\equiv \mathbf 1_{g}\: ({\rm mod}\: \q^2).
\end{equation}
\end{lemma} 
\begin{proof}
We follow closely the proof of Lemma 13.1 from \cite{BMS}. 
One verifies easily that $Sp(2g,\A, \q|\q^2)$ is indeed a subgroup 
of $Sp(2g,\A)$. It then suffices to show that all elements of the form 
$\left( \begin{matrix}
                  A & 0 \\
                  0  & (A^{-1})^{\top}
                 \end{matrix}
 \right)$, with $A\in GL(g,\A)$ satisfying 
$A\equiv \mathbf 1_{g} \: ({\rm mod } \: \q^2)$  
belong to $E(2g,\A,\q)$. Here the notation $A^{\top}$ stands 
for the transpose of the matrix $A$. 
Next, it suffices to verify this claim when 
$A$ is an elementary matrix, and hence when $A$ is in 
$GL(2,\A)$ and $g=2$. Taking therefore $A=\left( \begin{matrix}
                  1 & 0 \\
                  q_1q_2  & 1
                 \end{matrix}
 \right)$, where $q_1,q_2\in\q$, we can write: 
\begin{equation}
\left( \begin{matrix}
                  A & 0 \\
                  0  & (A^{-1})^{\top}
                 \end{matrix}
 \right)= 
\left( \begin{matrix}
                  1_{2} & \sigma A^{\top} \\
                  0  & 1_{2}
                 \end{matrix}
 \right)\left( \begin{matrix}
                  1_{2} & 0 \\
                  -\tau  & 1_{2}
                 \end{matrix}
 \right)\left( \begin{matrix}
                  1_{2} & \sigma \\
                  0  & 1_{2}
                 \end{matrix}
 \right)\left( \begin{matrix}
                  1_{2} & 0 \\
                  \tau  & 1_{2}
                 \end{matrix}
 \right),
\end{equation} 
where $\sigma=\left( \begin{matrix}
                 0 & q_1 \\
                 q_1  & 0
                 \end{matrix}
 \right)$ and $\tau=\left( \begin{matrix}
                 q_2 & 0 \\
                 0  & 0
                 \end{matrix}
 \right)$. 
\end{proof}

\vspace{0.2cm}
\noindent
Define now 
\begin{equation}
R_{ij}(q) 
= \left( \begin{matrix}
                  \mathbf 1_{g}+ qe_{ij} & 0 \\
                  0  & \mathbf 1_{g}-q e_{ji}
                 \end{matrix}
 \right), \: {\rm for }\:  i\neq j, 
\end{equation} 
and 
\begin{equation} 
N_{ii}(q)=  
\left( \begin{matrix}
                  \mathbf 1_{g}+ (q+q^2+q^3)e_{ii} & -q^2e_{ii} \\
                  q^2e_{ii}  & \mathbf 1_{g}-q e_{ii}
                 \end{matrix}
 \right).
\end{equation}

\begin{proof}[Proof of proposition \ref{lem H2epi}]
If $M$ is endowed with a structure of $G$-module we denote by $M_G$ the quotient module
 of co-invariants, namely the quotient of $M$ by the submodule  
generated by $\langle g\cdot x-x, g\in G, x\in M\rangle$. 

\vspace{0.2cm}\noindent 
It is known (see \cite{Brown}, VII.6.4) that an exact sequence of groups 
\[ 1\to K\to G \to Q\to 1\] 
induces the following 5-term exact sequence in homology 
(with coefficients in an arbitrary $G$-module $M$): 
\begin{equation}
H_2(G;M)\to H_2(Q;M_H)\to H_1(K)_{Q}\to H_1(G,M)\to H_1(Q,M_H) \to 0.
\end{equation}

\vspace{0.2cm}\noindent 
From the $5$-term exact sequence associated to the short exact sequence: 
\[1\to  Sp(2g,\A,\q)\to Sp(2g,\A)\to Sp(2g,\A/\q)\to 1\]
we deduce that the surjectivity of $p_\ast$ is equivalent to  Lemma~\ref{annulH1} hereafter.

\begin{lemma}\label{annulH1}
For any ideal $\q$, if $g\geq 3$ we have:  
\begin{equation}
H_1(Sp(2g,\A,\q))_{Sp(2g,\A/\q)} = 0.
\end{equation}
\end{lemma}
\begin{proof}[Proof of Lemma \ref{annulH1}]
For each pair of distinct integers 
$i,j\in\{1,\dots,g\}$ denote by $A_{ij}$ 
the symplectic matrix:

\begin{equation}
A_{ij}=\left( \begin{matrix}
                  \mathbf 1_{g}-e_{ij} & 0 \\
                   0   & \mathbf 1_g+e_{ji}
                 \end{matrix}
\right).
\end{equation}

\vspace{0.2cm}\noindent 
Then for each triple of distinct integers $(i,j,k)$ and $q\in \q$ the action 
of $A_{ij}$  by conjugacy on the generating matrices of 
$Sp(2g,\A,\q)$ is given by:  

\[
A_{ij} \cdot U_{jj}(q) = U_{ii}(q) U_{jj}(q) U_{ij}(q)^{-1}, 
\]
\[
A_{ki}\cdot U_{ij}(q)= U_{ij}(q) U_{jk}(q)^{-1}, \quad 
\]
\[
A_{ji}\cdot U_{ij}(q)= U_{ij}(q) U_{jj}(q)^{-2},
\]
\[
A_{ij}\cdot L_{jj}(q) = L_{ii}(q)L_{jj}(q)L_{ij}(q),
\]
\[  
A_{ik} \cdot L_{ij}(q) = L_{ij}(q), 
\]
\[A_{ij}\cdot  L_{ij}(q) =L_{ij}(q)L_{ii}(q)^2,
\]
where we used the notation $A\cdot U= AUA^{-1}$.  We also have: 
\[
A_{ij}\cdot R_{jk}(q) = R_{jk}(q)R_{ik}(q).
\]
Further the symplectic map $J=\left( \begin{matrix}
                  0 & -{\mathbf 1}_g \\
                  \mathbf 1_g  & 0
                 \end{matrix}
 \right)$ acts as follows: 

\[
 J \cdot U_{ij}(q) =-L_{ij}(q), \quad J \cdot U_{ii}(q) = -L_{ii}(q) .
\]

Denote by lower cases the classes of the maps 
$U_{ij}(q)$, $U_{ii}(q)$, $L_{ij}(q)$,  $L_{ii}(q)$, $R_{ij}(q)$ 
and $N_{ii}(q)$ in the  quotient 
$H_1(Sp(2g,\A,\q))_{Sp(2g,\A/\q)}$ of the abelianization 
$H_1(Sp(2g,\A,\q))$. By definition, the action of 
$Sp(2g,\A/\q)$ on $H_1(Sp(2g,\A,\q))_{Sp(2g,\A/\q)}$ is trivial.

Using the action of $J$ we obtain that  
$u_{ij}(q)+l_{ij}(q)=0$ in $H_1(Sp(2g,\A,\q))_{Sp(2g,\A/\q)}$, 
and hence  we can discard the generators $l_{ij}(q)$. 
As $g \geq 3$, from  the action of $A_{ki}$ on $u_{ij}(q)$ 
we derive that $u_{jk}(q)=0$, for every $j\neq k$. 
Using the action of $A_{ij}$ on $u_{jj}(q)$ we obtain that 
$u_{jj}(q)=0$ for every $j$. Further the action of $A_{ij}$ on $r_{jk}(q)$ 
yields $r_{ik}(q)=0$, for all $i\neq k$.

Consider now the symplectic matrix 
$B_{ii}=\left( \begin{matrix}
                   {\mathbf 1}_g & 0 \\
                    e_{ii}  & \mathbf 1_g  
                 \end{matrix}
 \right)$. Then 

\begin{equation}
B_{ii}\cdot U_{ii}(q)=\left( \begin{matrix}
                   {\mathbf 1}_g -q e_{ii} & q e_{ii} \\
                   -q e_{ii}  & \mathbf 1_g + q e_{ii}  
                 \end{matrix}
 \right)
\end{equation}
and hence 
\[
B_{ii}\cdot U_{ii}(q)
\equiv U_{ii}(q) L_{ii}(q)^{-1} N_{ii}(q)^{-1} \: ({\rm mod }\: \q^2).
\]

Recall that the elements of $Sp(2g,\A,\q^2)\subset Sp(2g,\A,\q|\q^2)$ 
can be written as products of the generators $U_{ij}(q)$ and $L_{ij}(q)$  
according to Lemma \ref{thetasubgroup}, 
whose images in the quotient $H_1(Sp(2g,\A,\q))_{Sp(2g,\A/\q)}$  vanish. 
Therefore $B_{ii}\cdot u_{ii}(q)= u_{ii}(q)+l_{ii}(q)-n_{ii}(q)$. 
This proves that $n_{ii}(q)=0$ in $H_1(Sp(2g,\A,\q))_{Sp(2g,\A/\q)}$. 
Consequently $H_1(Sp(2g,\A,\q))_{Sp(2g,\A/\q)}=0$, as claimed. 
\end{proof}

\end{proof}

\subsection{Nondivisibility of the universal symplectic central extension when restricted to level
subgroups} We will show that lemma \ref{putmanlemma} is sharp: 

\begin{lemma}\label{notputmanlemma}
Let $L\geq 2$ be an even number. The pull-back of the class of the universal central extension $c \in H^2(Sp(2g,\Z);\Z)$ to $Sp(2g,L)$  is not divisible by 2 if $4 \nmid L$.  
\end{lemma}
The proof is the result of discussions we had with D.~Benson about Lemma \ref{putmanlemma}.
\begin{proof}
Let $\iota:Sp(2g,2)\to Sp(2g,\Z)$ be the inclusion.  
We will prove  by contradiction that $\iota^*(c)\in H^2(Sp(2g,2);\Z)$ is not divisible by $2$. 
Suppose that there exists an extension $\widetilde{G}$ of $Sp(2g,2)$ 
whose class $d$ in $H^2(Sp(2g,2);\Z)$ satisfies $\iota^*(c)=2d$. 
Then we have a commutative diagram: 

\[
\begin{array}{clccccc}
1\to &\Z      & \to& \widetilde{G} & \to &  Sp(2g,2)& \to 1\\
     & \downarrow \times 2  & & \downarrow & & \downarrow \iota & \\
1\to &\Z & \to &\widetilde{Sp(2g,\Z)} &\to& Sp(2g,\Z)       & \to 1.
\end{array}
\]

Denote by $\widehat{Sp(2g,\Z)}$ the quotient of the universal central extension 
$\widetilde{Sp(2g,\Z)}$ by the square of the 
generator of the central $\Z$.

The composition of  group homomorphisms 
$\widetilde{G}\to \widetilde{Sp(2g,\Z)} \to \widehat{Sp(2g,\Z)}$ lifts the 
inclusion $\iota$ and factors through $\widetilde{G}/\Z=Sp(2g,2)$. Therefore 
the restriction of the extension $\widehat{Sp(2g,\Z)}$ to $Sp(2g,2)$ is split.

The image $H$ of $Sp(2g,2)$ within $\widehat{Sp(2g,\mathbb{Z})}$ by this section 
might not be a normal subgroup. However the group generated by squares of elements in  $H$
is a normal subgroup $K\lhd \widehat{Sp(2g,\mathbb{Z})}$. Moreover, $K$ is 
isomorphic to the group $Sp(2g,4)$, which is the group generated by the squares in $Sp(2g,2)$. 

By taking the quotient by the normal subgroup $K$ above we obtain a 
central extension: 
\[ 1\to \Z/2\Z \to \widehat{Sp(2g,\Z/4\Z)} \to Sp(2g,\Z/4\Z)  \to 1.\]
By the proof of Theorem \ref{tors-sympl0} this  exact sequence is nonsplit and by construction it splits when restricted to the subgroup $Sp(2g,2)/Sp(2g,4) =\mathfrak{sp}_{2g}(\mathbb{Z}/2\mathbb{Z})$, 
the symplectic Lie algebra over the field of two elements.

We now compute $H^2(Sp(2g,\Z/4\Z);\Z/2\Z)$ using the Leray-Serre spectral sequence of the extension:
\[
1 \to  \mathfrak{sp}_{2g}(\mathbb{Z}/2\mathbb{Z}) \to  Sp(2g,\Z/4\Z) \to  Sp(2g,\Z/2\Z) \to 1.
\]

By Steinberg's computation~\cite{Steinberg} $H_2(Sp(2g,\Z/2\Z),\Z/2\Z) =0$ for $g \geq 4$. Since symplectic groups are perfect we derive $H^2(Sp(2g,\Z/2\Z),\Z/2\Z)=0$.  
By Putman~\cite[Proof of Thm.G]{Pu1}, $H^1(Sp(2g,\Z/2\Z),\mathfrak{sp}_{2g}(\mathbb{Z}/2\mathbb{Z}))= 0$, 
so the only possibility is that $H^2(Sp(2g,\Z/4\Z);\Z/2\Z)) = H^0(Sp(2g,\Z/2\Z); H^2 (\mathfrak{sp}_{2g}(\mathbb{Z}/2\mathbb{Z});\mathbb{Z}/2\mathbb{Z} )) \neq 0$. But this means  that the above nonsplit extension of $Sp(2g,\Z/4\Z)$  can not split when restricted to $\mathfrak{sp}_{2g}(\mathbb{Z}/2\mathbb{Z})$, a contradiction.

The proof for even $L$ with $4 \nmid L$ is similar and we skip the details.

\end{proof}  

\begin{remark}
D.~Benson informed us that the largest subgroup of $\mathfrak{sp}_{2g}(\mathbb{Z}/2\mathbb{Z})$
on which the above extension splits has index $2^{g+1}$. 
\end{remark}

{
\small

\bibliographystyle{plain}

}

\end{document}